\documentclass[12pt,leqno]{amsart}
\usepackage{color,amsmath,amsfonts,amssymb,amscd,amsthm,amsbsy,upref, %
}
\usepackage{setspace}
\usepackage{enumerate}
\allowdisplaybreaks

\usepackage[bookmarks,bookmarksopen=false,
                  pdfauthor={Autor},%
                  pdftitle={Titel},%
                  colorlinks=true,%
                  linkcolor=blue,%
                  citecolor=blue,%
                  urlcolor=blue,%
                ]{hyperref}

\numberwithin{equation}{section}
\def\hangbox to #1 #2{\vskip3pt\hangindent #1\noindent \hbox to #1{#2}$\!\!$}

\newtheorem{thm}{Theorem}[section]

\newtheorem{lem}[thm]{Lemma}
\newtheorem{cor}[thm]{Corollary}
\newtheorem{conj}[thm]{Conjecture}

\newtheorem{example}[thm]{Example}
\newtheorem{prop}[thm]{Proposition}
\theoremstyle{definition}

\theoremstyle{remark}

\def\N{{\mathbb N}}

\def\E{{\mathbb E}}

\def\sfrac#1#2{\kern.1em\raise.5ex\hbox{$#1$}
        \kern-.1em/\kern-.05em\lower.25ex\hbox{$#2$}}

\def\vp{\varepsilon}

\newcommand{\ap}{\alpha}

\newcommand{\fw}{\text{\fw}}

\newcommand{\trace}{{\rm trace}}

\usepackage{color}

\definecolor{darkviolet}{rgb}{0.58,0,0.83} 

\begin{document}

\title{Quantitative bounds for unconditional pairs of frames}

\author{ Peter Balazs}
\address{Acoustics Research Institute\\
 Austrian Academy of Sciences\\
 Wohllebengasse 12-14, 1040 Vienna, Austria}
 \email{peter.balazs@oeaw.ac.at}

\author{ Daniel Freeman}
\address{Department of Mathematics and Statistics\\
St Louis University\\
St Louis MO 63103  USA
} \email{daniel.freeman@slu.edu}

\author{ Roxana Popescu}
\address{Department of Mathematics\\
University of Pittsburgh\\
Pittsburgh, PA 15260  USA}
\email{rop42@pitt.edu}

\author{ Michael Speckbacher}
\address{Faculty of Mathematics\\
University of Vienna\\
Oskar-Morgenstern-
Platz 1, A-1090 Vienna, Austria}
\email{michael.speckbacher@univie.ac.at}

 \begin{abstract}
 
 We formulate a quantitative finite-dimensional conjecture about frame multipliers and prove that it is equivalent to Conjecture 1 in \cite{SB2}.
 We then present solutions to the conjecture for certain classes of frame multipliers.  In particular, we prove that there is a universal constant $\kappa>0$ so that for all $C,\beta>0$ and $N\in\N$ the following is true.  Let $(x_j)_{j=1}^N$ and $(f_j)_{j=1}^N$ be sequences in a finite dimensional Hilbert space  which satisfy $\|x_j\|=\|f_j\|$ for all $1\leq j\leq N$ and
 $$\Big\|\sum_{j=1}^N \vp_j\langle x,f_j\rangle  x_j\Big\|\leq C\|x\|, \qquad\textrm{ for all $x\in \ell_2^M$ and $|\vp_j|=1$}.
$$
If the frame operator for $(f_j)_{j=1}^N$  has eigenvalues $\lambda_1\geq...\geq\lambda_M$ and $\lambda_1\leq \beta M^{-1}\sum_{j=1}^M\lambda_j$ then $(f_j)_{j=1}^N$ has Bessel bound $\kappa \beta^2 C$.  The same holds for $(x_j)_{j=1}^N$.
\end{abstract}

\thanks{The first author was supported by P 34624 and P 34922 funded by the  Austrian Science Fund (FWF).The second author was supported by grant 706481 from the Simons Foundation and grant 2154931 from the National Science Foundation. The fourth author was supported by the Austrian Science Fund (FWF) through the projects Y-1199 and J-4254.}

\maketitle

\section{Introduction}
 A {\em frame} for a finite dimensional or infinite dimensional separable Hilbert space $H$ is a sequence of vectors
$(x_j)_{j=1}^N\subset H$ (where $N\in\N$ or $N=\infty$) for which there exist constants
$0< A\leq B$,  called {\em frame bounds}, such that
\begin{equation}\label{E:frame}
A\|x\|^2\leq \sum_{i=1}^N |\langle x, x_j\rangle|^2\leq
B\|x\|^2,\qquad \textrm{ for all }x\in H.
\end{equation}
The ratio $B/A$ is called the  {\em condition number} of $(x_j)_{j=1}^N$.
We say that  $(x_j)_{j=1}^N$ is {\em Bessel} if it satisfies the upper bound of \eqref{E:frame} for some $B<\infty$ and call $B$ a {\em Bessel bound}.   A frame is called \emph{tight} if it has condition number $1$.  The {\em analysis operator} of $X=(x_j)_{j=1}^N$ is the map $U_X:H\rightarrow \ell_2^N$ given by $U_X(x)=(\langle x,x_j\rangle)_{j=1}^N$ for all $x\in H$ and the {\em frame operator} of $X=(x_j)_{j=1}^N$ is the positive operator $S_X:H\rightarrow H$ given by $S_X=U_X^* U_X$.  Note that $(x_j)_{j=1}^N$ has Bessel bound $B$ if and only if $\|S_X\|=\|U_X\|^2\leq B$.  

For a choice of $\mathcal{E}=(\vp_j)_{j=1}^N$ with $|\vp_j|=1$ for all $1\leq j\leq N$ we let $D_\mathcal{E}:\ell_2^N\rightarrow \ell_2^N$ be the map $D_\mathcal{E}(b_j)_{j=1}^N=(\vp_j b_j)_{j=1}^N$.
It immediately follows that if $X=(x_j)_{j=1}^N$ and $F=(f_j)_{j=1}^N$ are both sequences in $H$ with Bessel bound $B$ and analysis operators $U_X$ and $U_F$ then
\begin{equation}
\|U^*_X D_\mathcal{E} U_F x\|=\Big\|\sum_{j=1}^N \vp_j\langle x, f_j\rangle x_j\Big\|\leq B\|x\|, \qquad\textrm{for all $x\in H$ and all  $\mathcal{E}=(\vp_j)_{j=1}^N$}.    
\end{equation}
This idea can be generalized further through the introduction of frame multipliers \cite{B1,SB1}.  Let $X=(x_j)_{j=1}^N$ and $F=(f_j)_{j=1}^N$ be sequences in a Hilbert space $H$, and 
let $m=(m_j)_{j=1}^N$ be a sequence of scalars called the {\it symbol}. 
The corresponding frame multiplier $M_{m,X, F}:H\rightarrow H$ is given by 
\begin{equation} \label{multdef}
 M_{m,X,F}x=\sum_{j=1}^N m_j \langle x,f_j \rangle x_j, \qquad\textrm{for all $x\in H$.}
\end{equation} 
  These operators are closely related to the concept of weighted frames \cite{BAG}, i.e. sequences $(d_j x_j)$, where $(d_j)$ is a sequence of scalars and $(x_j)$ is a sequence of vectors.  Note that if $(d_j)_{j=1}^N$ is a sequence of non-zero scalars then the frame multipliers $M_{m,(d_j x_j), (\overline{d}_j^{-1}f_j)}$ and $M_{m,(x_j), (f_j)}$ are equal.  However, the sequences $(d_j x_j)_{j=1}^n$ and $(\overline{d}_j^{-1}f_j)_{j=1}^N$ may have very different frame bounds from $(x_j)_{j=1}^N$ and $(f_j)_{j=1}^N$.
The following conjectures that if  we are given an unconditionally convergent multiplier then we can shift the weights to get two Bessel sequences.

\begin{conj}[\cite{SB2}]\label{C:SB}
Let $M_{m,X,F}$ be an unconditionally convergent multiplier on a separable Hilbert space $H$.   Then there exists sequences of scalars $(c_j)_{j=1}^\infty$ and $(d_j)_{j=1}^\infty$ such that $c_j \overline{d_j}=m_j$ for all $j\in \N$ and both $(c_j x_j)_{j=1}^\infty$ and $(d_j f_j)_{j=1}^\infty$ are Bessel.
\end{conj}

This idea of shifting weights is also considered in \cite{HLLL} for the case where the multiplier is constant $1$ and the sequences satisfy a reproducing formula.  Suppose that $(x_j)_{j=1}^\infty$ and $(f_j)_{j=1}^\infty$ are sequences in a Hilbert space $H$ such that $x=\sum \langle x,f_j\rangle x_j$, and the series converges unconditionally  for all $x\in H$.   Then there exists a sequence $(d_j)_{j=1}^\infty$ such that $(d_j x_j)_{j=1}^\infty$ and $(\overline{d_j}^{-1} f_j)_{j=1}^\infty$ are both frames of $H$ if and only if the induced operator valued map ${\mathfrak M}^{F,X} (a_j)_{j=1}^\infty = \sum a_j f_j\otimes x_j$ is a completely bounded map between the $C^*$-algebra $\ell_\infty$ and the $C^*$-algebra $B(H)$ (where $B(H)$ is the space of bounded operators on $H$) \cite{HLLL}. 

Many problems for infinite-dimensional Hilbert spaces have corresponding quantitative problems for finite-dimensional Hilbert spaces.  Notably, the Kadison-Singer Problem was a famous and long open question about operators on infinite dimensional Hilbert spaces which was shown to be equivalent to  the Feichtinger Conjecture \cite{CCLV}, the Paving Conjecture \cite{A}, Weaver's Conjecture \cite{W}, and the Bourgain-Tzafriri Conjecture \cite{BT}. 
 Marcus, Spielman, and Srivastava \cite{MSS} solved the Kadison-Singer Problem by proving a very strong quantitative and finite-dimensional theorem which directly implied Weaver's Conjecture and has since been applied to solve many other problems in applied harmonic analysis and approximation theory \cite{NOU}\cite{FS}\cite{LT}\cite{DKU}.
 In Section \ref{S:fm} we show that Conjecture \ref{C:SB} is equivalent to the following  quantitative and finite-dimensional  conjecture.

\begin{conj}\label{C:SB quant}
There exists a universal constant $\kappa>0$ so that the following holds.  Let $C>0$ and let $M_{(m_j)_{j=1}^N,(x_j)_{j=1}^N,(f_j)_{j=1}^N}$ be a multiplier on a finite dimensional Hilbert space $H$ such that
\begin{equation}\label{E:unc}
 \Big\|\sum_{j=1}^N \vp_j m_j\langle x,f_j \rangle  x_j\Big\|\leq C\|x\|, \qquad\textrm{for all $x\in H$ and $|\vp_j|=1$}.
\end{equation}
Then there exists sequences of constants $(c_j)_{j=1}^N$ and $(d_j)_{j=1}^N$ such that $c_j \overline{d_j}=m_j$ for all $1\leq j\leq N$ and both $(c_j x_j)_{j=1}^N$ and $(d_j f_j)_{j=1}^N$ are  $C\kappa$-Bessel.
\end{conj}

By expressing Conjecture \ref{C:SB} in a quantitative and finite dimensional way, we hope to open the problem to new methods and techniques.
Conjecture \ref{C:SB}  has been solved for a large number of important classes of sequences \cite{SB2}.  Likewise, we will solve Conjecture \ref{C:SB quant}  for certain important cases. Our results are distinctly different from what has been done in infinite-dimensions, and we will make use of probabilistic methods which are inherently finite-dimensional.  The following theorem solves Conjecture \ref{C:SB} in the case where the largest eigenvalue of the frame operator is proportional to the average of the eigenvalues.

\begin{thm}\label{T:intro}
 Let $C,\beta>0$ and $N,M\in\N$.  Suppose that $(x_j)_{j=1}^N$ and $(f_j)_{j=1}^N$ are  sequences in an $M$-dimensional Hilbert space  which satisfy $\|x_j\|=\|f_j\|$ for all $1\leq j\leq N$ and
 \begin{equation}\label{E:unc T:intro}
  \Big\|\sum_{j=1}^N \vp_j\langle x,f_j \rangle  x_j\Big\|\leq C\|x\|, \qquad\textrm{ for all $x\in \ell_2^M$ and $|\vp_j|=1$}.
 \end{equation}
If the frame operator for $(f_j)_{j=1}^N$  has eigenvalues $\lambda_1\geq...\geq\lambda_M$ satisfying $\lambda_1\leq \frac{\beta}{ M}\sum_{j=1}^M\lambda_j$ then $(f_j)_{j=1}^N$ has Bessel bound $\frac{27}{4} K_{1}^{-4} \beta^2 C$, where $K_{1}$ is the universal constant given by Khintchine's Inequality (Theorem~\ref{thm:khintchine}).  The same holds for $(x_j)_{j=1}^N$.
\end{thm}

Note that if $(f_j)_{j=1}^N$ is a frame of an $M$-dimensional Hilbert space and $\lambda_1\geq...\geq\lambda_M$ are the eigenvalues of the frame operator of $(f_j)_{j=1}^N$ then $(f_j)_{j=1}^N$ has condition number $\lambda_1/\lambda_M$.  Thus, if $(x_j)_{j=1}^N$ and $(f_j)_{j=1}^N$ are frames with condition number $\beta$ and  $\|x_j\|=\|f_j\|$ for all $1\leq j\leq N$ then $(x_j)_{j=1}^N$ and $(f_j)_{j=1}^N$ both have Bessel bound $\frac{27}{4} K_{1}^{-4} \beta^2 C$ where $C$ satisfies \eqref{E:unc T:intro}.
This gives the following  corollary for pairs of equi-norm tight frames.

\begin{cor} \label{C:intro}
Let $(x_j)_{j=1}^N$ and $(f_j)_{j=1}^N$ be tight frames for a finite dimensional Hilbert space with $\|x_j\|=\|f_j\|$ for all $1\leq j\leq N$.  Let $C>0$ be the least constant such that
$$
 \Big\|\sum_{j=1}^N \vp_j\langle x,f_j\rangle  x_j\Big\|\leq C\|x\|, \qquad\textrm{for all $x\in \ell_2^M$ and $|\vp_j|=1$}.
$$
Then the tight frames $(x_j)_{j=1}^N$ and $(f_j)_{j=1}^N$ have the same frame bound $B=M^{-1}\sum_{j=1}^N\|x_j\|^2$ and  $C\leq B\leq \frac{27}{4} K_{1}^{-4}C$.
\end{cor}

In Theorem \ref{T:intro} we made an assumption about the eigenvalues of the frame operators for $(x_j)_{j=1}^N$ and $(f_j)_{j=1}^N$. In the following theorem we do not assume anything about the frame structure of $(x_j)_{j=1}^N$ and $(f_j)_{j=1}^N$ but we require a uniform lower bound $\|x_j\|\|f_j\|\geq b$ for all $1\leq j\leq N$.
This gives a quantitative version of one direction of Proposition 1.1 in \cite{SB2}.

\begin{prop}\label{P:Par_intro}
Let $b,C>0$ and $N\in\N\cup\{\infty\}$.  Let $(x_j)_{j=1}^N,(f_j)_{j=1}^N$ be sequences in a Hilbert space $H$ such that $\|x_j\|\|f_j\|\geq b$  for all $1\leq j\leq N$ and that
$$\Big\|\sum_{j=1}^N \vp_j\langle x, f_j\rangle x_j\Big\|\leq C\|x\|, \qquad\textrm{for all $x\in H$ and $|\vp_j|=1$}.
$$
Then $(d_jx_j)_{j=1}^N$ and $(d_j^{-1}f_j)_{j=1}^N$ have Bessel bound $b^{-1}C^2$ where $d_j=\|x_j\|^{-1/2}\|f_j\|^{1/2}$ for all $1\leq j\leq N$.
\end{prop}

Note that the Bessel bound given in Proposition \ref{P:Par_intro} scales as $C^2$ instead of $C$ as in Conjecture \ref{C:SB quant}.
Furthermore, we give an example in Section \ref{S:Par} which shows that the scaling of $C^2$ is necessary here.  However, this does not contradict Conjecture \ref{C:SB quant} as the choice of $(d_j)_{j=1}^N$ specified in Proposition \ref{P:Par_intro} is not necessarily the optimal choice for Conjecture \ref{C:SB quant}.

The paper is organized as follows.  In Section \ref{S:fm} we prove that Conjecture \ref{C:SB} is equivalent to Conjecture \ref{C:SB quant}.  In Section \ref{S:Par} we use the Parallelogram Law to give a short proof of Proposition \ref{P:Par_intro}.  We prove our main results, including Theorem \ref{T:intro}, in Section \ref{S:equinorm}.

\section{Frame multipliers in finite dimensions}\label{S:fm}

Our goal for this section is to prove that Conjecture \ref{C:SB} on unconditionally convergent frame multipliers for infinite dimensional Hilbert spaces is equivalent to  Conjecture \ref{C:SB quant} on uniform quantitative bounds for frame multipliers on finite dimensional Hilbert spaces. 

\begin{thm}
The following are equivalent.
\begin{enumerate}
    \item  For every unconditonally convergent frame multiplier $M_{(m_j)_{j=1}^\infty,(x_j)_{j=1}^\infty,(f_j)_{j=1}^\infty}$  on a separable Hilbert space $H$ there exists sequences of constants $(c_j)_{j=1}^\infty$ and $(d_j)_{j=1}^\infty$ such that $c_j \overline{d_j}=m_j$ for all $j\in \N$ and both $(c_j x_j)_{j=1}^\infty$ and $(d_j f_j)_{j=1}^\infty$ are Bessel.\\
    \item  There exists a universal constant $\kappa>0$ so that the following holds.  Let $C>0$ and let $M_{(m_j)_{j=1}^n,(x_j)_{j=1}^n,(f_j)_{j=1}^n}$ be a multiplier on a finite dimensional Hilbert space $H$ such that
$$
 \Big\|\sum_{j=1}^n \vp_j m_j\langle x,f_j \rangle  x_j\Big\|\leq C\|x\|, \qquad\textrm{for all $x\in H$ and $|\vp_j|=1$}.
$$
Then there exists sequences of scalars $(c_j)_{j=1}^n$ and $(d_j)_{j=1}^n$ such that $c_j \overline{d_j}=m_j$ for all $1\leq j\leq n$ and both $(c_j x_j)_{j=1}^n$ and $(d_j f_j)_{j=1}^n$ are $C\kappa$-Bessel.
\end{enumerate}
\end{thm}

\begin{proof}
We first assume that (2) is true and will prove that (1) is true.  Let $M_{(m_j)_{j=1}^\infty,(x_j)_{j=1}^\infty,(f_j)_{j=1}^\infty}$ be an unconditionally convergent frame multiplier on a separable Hilbert space $H$.  That is, 
\begin{equation}
    \sum_{j=1}^\infty \vp_j m_j \langle x,f_j\rangle x_j\hspace{1cm}\textrm{ converges for all $x\in H$ and $|\vp_j|=1$.}
\end{equation}
For each $m,n\in\N$ and $(\vp_j)_{j=m}^n$ with $|\vp_j|=1$ for all $m\leq j\leq n$ we let $T_{(\vp)_{j=m}^n}$ be the finite rank operator on $H$ defined by $T_{(\vp)_{j=m}^n}(x)=\sum_{j=m}^n \vp_j m_j \langle x,f_j\rangle x_j$ for all $x\in H$.  Let $x\in H$ and for the sake of contradiction we assume that $\sup_{ (\vp_j)_{j=m}^n} \|T_{(\vp)_{j=m}^n}(x)\|=\infty$.  By piecing finite sequences together, we can create an infinite sequence $(\vp_j)_{j=1}^\infty$ such that $\sup_{m\leq n}\|T_{(\vp)_{j=m}^n}(x)\|=\infty$.  This contradicts that $\sum_{j=1}^\infty \vp_j m_j \langle x,f_j\rangle x_j$ converges.  Hence, for all $x\in H$ there exists $C_x>0$ so that $\|T_{(\vp)_{j=m}^n}(x)\|\leq C_x \|x\|$ for all $(\vp)_{j=m}^n$.  By the Uniform Boundedness Principle there exists a uniform constant $C>0$ so that $\|T_{(\vp)_{j=m}^n}(x)\|\leq C \|x\|$ for all $(\vp)_{j=m}^n$ and all $x\in H$.  Thus, we have for all $N\in\N$ that 
\begin{equation}
    \Big\|\sum_{j=1}^N \vp_j m_j\langle x,f_j\rangle  x_j\Big\|\leq C\|x\|, \qquad\textrm{for all $x\in \text{span}_{1\leq j\leq N}x_j$ and $|\vp_j|=1$}.
\end{equation}
By (2), there exists $(c_{N,j})_{j=1}^n$ and $(d_{N,j})_{j=1}^n$ such that $c_{N,j} \overline{d_{N,j}}=m_j$ for all $j\in \N$ and both $(c_{N,j} x_j)_{j=1}^N$ and $(d_{N,j} f_j)_{j=1}^N$ are $\kappa C$-Bessel.  Without loss of generality, we assume that $m_j\neq0$, $x_j\neq0$, and $f_j\neq 0$ for all $j\in\N$.  Thus, we have that $|c_{N,j}|\leq \kappa C \|x_j\|^{-1}$ and $|d_{N,j}|\leq \kappa C \|f_j\|^{-1}$ for all $j\in\N$.  As $c_{N,j} \overline{d_{N,j}}=m_j$, we have for all $j\in \N$ that

\begin{equation}\label{E:cd}
     |m_j|\kappa^{-1} C^{-1} \|f_j\| \leq |c_{N,j}|\leq \kappa C \|x_j\|^{-1},\hspace{.1cm}\textrm{ and }\hspace{.1cm}
|m_j|\kappa^{-1} C^{-1} \|x_j\|  \leq |d_{N,j}|\leq \kappa C \|f_j\|^{-1}.
\end{equation}
Thus, for each $j\in\N$ we have positive uniform upper and lower bounds on $(c_{N,j})_{N=1}^\infty$ and $(d_{N,j})_{N=1}^\infty$. After passing to a subsequence, we may assume that there exists $(c_j)_{j=1}^\infty$ and $(d_j)_{j=1}^\infty$ so that $\lim_{N\rightarrow\infty}c_{N,j}=c_j$ and $\lim_{N\rightarrow\infty}d_{N,j}=d_j$ for all $j\in\N$.  By \eqref{E:cd} we have that $c_j$ and $d_j$ are non-zero and that $c_j\overline{d_j}=m_j$ for all $j\in \N$.  For all $n\leq N$ we have that $(c_{N,j} x_j)_{j=1}^n$ and $(d_{N,j} f_j)_{j=1}^n$ are $\kappa C$-Bessel.  Thus by taking the limit, we have for all $n\in\N$ that $(c_j x_j)_{j=1}^n$ and $(d_j f_j)_{j=1}^n$ are $\kappa C$-Bessel.  Hence, $(c_j x_j)_{j=1}^\infty$ and $(d_j f_j)_{j=1}^\infty$ are $\kappa C$-Bessel.  This proves (1).

We now assume that (2) is false.  Thus, for all $k\in\N$ there exists $C_k>0$ and a multiplier 
 $M_{(m_{k,j})_{j=1}^{n_k},(x_{k,j})_{j=1}^{n_k},(f_{j_k})_{j=1}^{n_k}}$  on a finite dimensional Hilbert space $H_k$ such that
$$
 \Big\|\sum_{j=1}^{n_k} \vp_j m_{k,j}\langle x,f_{k,j} \rangle  x_{k,j}\Big\|\leq C_k\|x\|, \qquad\textrm{for all $x\in H_k$ and $|\vp_j|=1$},
$$
but that for all sequences of scalars $(c_j)_{j=1}^{n_k}$ and $(d_j)_{j=1}^{n_k}$ such that $c_j \overline{d_j}=m_j$ for all $1\leq j\leq n_k$ we have that either $(c_j x_{k,j})_{j=1}^{n_k}$ or $(d_j f_{k,j})_{j=1}^{n_k}$ is not $k C_k$-Bessel.  By scaling both  $(x_{k,j})_{j=1}^{n_k}$ and $(f_{j_k})_{j=1}^{n_k}$ by $C_k^{-1/2}$ we have that  $M_{(m_{k,j})_{j=1}^{n_k},(C_k^{-1/2}x_{k,j})_{j=1}^{n_k},(C_k^{-1/2}f_{j_k})_{j=1}^{n_k}}$ is a multiplier on  $H_k$ such that 
$$
 \Big\|\sum_{j=1}^{n_k} \vp_j m_{k,j}\langle x,C_k^{-1/2}f_{k,j} \rangle  C_k^{-1/2}x_{k,j}\Big\|\leq \|x\|, \qquad\textrm{for all $x\in H_k$ and $|\vp_j|=1$},
$$
but that for all sequences of scalars $(c_j)_{j=1}^{n_k}$ and $(d_j)_{j=1}^{n_k}$ such that $c_j \overline{d_j}=m_j$ for all $1\leq j\leq k$ we have that either $\big(c_j C_k^{-1/2} x_{k,j}\big)_{j=1}^{n_k}$ or $\big(d_j C_k^{-1/2} f_{k,j}\big)_{j=1}^{n_k}$ is not $k$-Bessel. 

We now let $H=\oplus_{k=1}^\infty H_k$ and let $M_{(m_{j})_{j=1}^{\infty},(x_{j})_{j=1}^{\infty},(f_{k,j})_{j=1}^{\infty}}$ be the multiplier on $H$ which is the direct sum of the multipliers $\big(M_{(m_{k,j})_{j=1}^{n_k},(C_k^{-1/2}x_{k,j})_{j=1}^{n_k},(C_k^{-1/2}f_{k,j})_{j=1}^{n_k}}\big)_{k=1}^\infty$. That is, we enumerate by  $t\in\N$, where, if $k\in\N$, and $1\leq j\leq n_{k}$ are such that   $t=\sum_{i=1}^{k-1} n_i+j$ then $m_t=m_{k,j}$,  $x_t=C_k^{-1/2}x_{k,j}$ and $f_t=C_k^{-1/2}f_{j_k}$.
Let $x\in H$ and let $(\vp_t)_{t=1}^\infty$ be a sequence of scalars with  $\vp_t=\pm1$ for all $t\in\N$.  For each $k\in\N$ and $1\leq j\leq n_k$ we let $\vp_{k,j}=\vp_t$ where $t=\sum_{i=1}^{k-1} n_i+j$.
Thus, we have that

\begin{align*}
\Big\|\sum_{t=1}^{\infty} \vp_t m_t\langle x,f_t\rangle  x_t\Big\|^2
&= \sum_{k=1}^\infty\Big\|\sum_{j=1}^{n_k} \vp_{k,j} m_{k,j}\langle P_{H_k}x,C_k^{-1/2}f_{k,j} \rangle  C_k^{-1/2}x_{k,j}\Big\|^2\\
&\leq \sum_{k=1}^\infty \|P_{H_k}x\|^2=\|x\|^2.
\end{align*}
Thus, the multiplier $M_{(m_{t})_{t=1}^{\infty},(x_{t})_{t=1}^{\infty},(f_{t})_{t=1}^{\infty}}$ is unconditionally convergent.  However, there does not exist scalars $(c_t)_{t=1}^\infty$ and $(d_t)_{t=1}^\infty$ such that $c_t \overline{d_t}=m_t$ for all $t\in \N$ and both $(c_t x_t)_{t=1}^\infty$ and $(d_t f_t)_{t=1}^\infty$ are Bessel.  Thus we have that (1) is false.

\end{proof}

\section{Lower bounds and the parallelogram law}\label{S:Par}

The parallelogram law states that if $(x_j)_{j=1}^N$ is a sequence of vectors in a Hilbert space then 
\begin{equation} \label{eq:paralleogram} 2^{-N}\sum_{\vp_j=\pm1} \Big\|\sum_{j=1}^N \vp_j x_j \Big\|^2 =\sum_{j=1}^N  \|x_j\|^2.
\end{equation}
That is, if $(\vp_j)_{j=1}^N$ is a sequence of independent mean-$0$ random variables with $|\vp_j|=1$ then we have the following formula for the expectation.
\begin{equation}\label{eq:expectation}
\E \Big\|\sum_{j=1}^N \vp_j x_j \Big\|^2 =\sum_{j=1}^N \|x_j\|^2.
\end{equation}
Note that Conjecture \ref{C:SB quant} concerns pairs of vectors $(x_j)_{j=1}^N$ and $(f_j)_{j=1}^N$ such that the inequality $\|\sum_{j=1}^N \vp_j \langle x,f_j\rangle x_j\|\leq C\|x\|$ is satisfied for all vectors  $x$ and all $|\vp_j|=1$.  This inequality naturally lends itself to the parallelogram law.  We now restate and prove Proposition \ref{P:Par_intro} from the introduction.

\begin{prop}\label{P:Par}
Let $b,C>0$ and $N\in\N\cup\{\infty\}$.  Let $(x_j)_{j=1}^N,(f_j)_{j=1}^N$ be sequences in a Hilbert space $H$ such that $\|x_j\|\|f_j\|\geq b$  for all $1\leq j\leq N$ and that
\begin{equation}\label{E:unc par}
\Big\|\sum_{j=1}^N \vp_j\langle x, f_j\rangle x_j\Big\|\leq C\|x\|, \qquad\textrm{for all $x\in H$ and $|\vp_j|=1$}.
\end{equation}
Then $(d_jx_j)_{j=1}^N$ and $(d_j^{-1}f_j)_{j=1}^N$ have Bessel bound $b^{-1}C^2$ where $d_j=\|x_j\|^{-1/2}\|f_j\|^{1/2}$ for all $1\leq j\leq N$.
\end{prop}
\begin{proof}
We first assume that $N\in\N$ is finite.
Let  $(\vp_j)_{j=1}^N$ be a sequence of independent mean-$0$ random variables with $|\vp_j|=1$.  Let $x\in H$.
We have that,
\begin{align*}
C^2\|x\|^2&\geq \E \Big\|\sum_{j=1}^N \vp_j\langle x, f_j\rangle  x_j\Big\|^2\\
&= \E \Big\|\sum_{j=1}^N \vp_j\langle x, 
d_j^{-1}f_j\rangle d_j x_j\Big\|^2\\
&= \sum_{j=1}^N |\langle x,
d_j^{-1}f_j \rangle|^2 d_j^2\|x_j\|^2\qquad\textrm{ by the parallelogram law,}\\
&=\sum_{j=1}^N |\langle x, 
d_j^{-1}f_j\rangle|^2 \|x_j\|\|f_j\|\\
&\geq b \sum_{j=1}^N |\langle x,
d_j^{-1}f_j\rangle|^2.
\end{align*}
This gives that $(d_j^{-1}f_j)_{j=1}^N$ has Bessel bound $b^{-1}C^2$.
For all $|\vp_j|=1$, the adjoint of the operator $S(x)=\sum_{j=1}^N \vp_j \langle x, f_j\rangle x_j$ is the operator  $S^*(f)=\sum_{j=1}^N \overline{\vp_j}\langle f, x_j\rangle f_j$.  Thus, the roles of $(f_j)_{j=1}^N$ and $(x_j)_{j=1}^N$ may be interchanged.  The same argument we used for  $(\overline{d_j}^{-1}f_j)_{j=1}^N$ now proves that $(d_j x_j)_{j=1}^N$ has Bessel bound $b^{-1}C^2$.

For the case $N=\infty$, we have that $(d_jx_j)_{j=1}^\infty$ and $(d_j^{-1}f_j)_{j=1}^\infty$ have Bessel bound $b^{-1}C^2$ if and only if $(d_jx_j)_{j=1}^n$ and $(d_j^{-1}f_j)_{j=1}^n$ have Bessel bound $b^{-1}C^2$ for all $n\in\N$.  Thus, the infinite case follows from the finite case.

\end{proof}

For the infinite case, it was previously known that if $(x_j)_{j=1}^\infty$ and $(f_j)_{j=1}^\infty$ satisfy the hypothesis of Proposition \ref{P:Par} then $(d_jx_j)_{j=1}^n$ and $(d_j^{-1}f_j)_{j=1}^n$ are both Bessel \cite{SB2}. The contribution of Proposition \ref{P:Par} is that it provides an explicit Bessel bound.

Note that the Bessel bound given in Proposition \ref{P:Par} scales as $C^2$ instead of $C$ as in Conjecture \ref{C:SB quant}.
The following example shows that this is necessary.  

\begin{example}\label{Ex}
Let $(e_j)_{j=1}^N$ be the unit vector basis of $\ell_2^N$.  We let $x_j=e_j$ and $f_j=e_1$ for all $1\leq j\leq N$.  Then we have that,
\begin{equation}\label{E:example}
\Big\|\sum_{j=1}^N \vp_j\langle x, e_1\rangle e_j\Big\|\leq N^{1/2}\|x\|, \qquad\textrm{for all $x\in \ell_2^N$ and $|\vp_j|=1$}.
\end{equation}
Note that equality in \eqref{E:example} is achieved  for $x=e_1$.  We have the values $b=1$ and $C=N^{1/2}$ for Proposition \ref{P:Par}, which gives that $(x_j)_{j=1}^N$ and $(f_j)_{j=1}^N$ must have Bessel bound $b^{-1}C^2=N$. Furthermore, $N$ is exactly the Bessel bound of $(f_j)_{j=1}^N=(e_1)_{j=1}^N$.
\end{example}

We have that the sequence of pairs $(e_j,e_1)_{j=1}^N$ in Example \ref{Ex} satisfies the unconditionality inequality with constant $N^{1/2}$ and yet $(e_1)_{j=1}^N$ has Bessel bound $N$.  However, we can shift weights so that $(N^{1/4}e_j)_{j=1}^N$ and $(N^{-1/4}e_1)_{j=1}^N$ each have Bessel bound $N^{1/2}$.  This shows that it may be necessary to shift weights to minimize the maximum of the Bessel bounds of $(d_jx_j)_{j=1}^N$ and $(d_j^{-1}f_j)_{j=1}^N$ even when $\|x_j\|=\|f_j\|$ for all $1\leq j\leq N$.

\section{Pairs of equi-norm frames}\label{S:equinorm}

 The flexibility of choosing a sequence of scalars $(d_j)_{j=1}^N$ is the most challenging aspect of Conjecture \ref{C:SB quant}.  One reason for the difficulty is that the Bessel bounds for $(d_j x_j)_{j=1}^N$ and $(d_j^{-1} f_j)_{j=1}^N$ are  global properties that apply to all $x\in H$, where as modifying the value for $d_k$ for some fixed $1\leq k\leq N$ makes a local change in one dimension for the frame operators for $(d_j x_j)_{j=1}^N$ and $(d_j^{-1} f_j)_{j=1}^N$.  Intuitively, a common problem in many different areas of mathematics is that it is difficult to optimize a global property through local modifications.  Another difficulty is that the optimal Bessel bound of the sequence $(d_j x_j)_{j=1}^N$ is a continuous function of the variables $(d_j)_{j=1}^N$, but it is not a smooth function of $(d_j)_{j=1}^N$. Indeed, for the simplest case where $(x_j)_{j=1}^N$ is an orthonormal basis, we have that the optimal Bessel bound for $(d_j x_j)_{j=1}^N$ is $B=\max_{1\leq j\leq N}|d_j|$.
 In Proposition \ref{P:Par} we explicitly choose  $d_j=\|x_j\|^{-1/2}\|f_j\|^{1/2}$ by using the local hypothesis $\|x_j\|\|f_j\|\geq b$  for all $1\leq j\leq N$.  This works well for large $b>0$, but (as shown in Example \ref{Ex}) this choice of $(d_j)_{j=1}^N$ may yield very large Bessel bound when $b$ is small relative to the unconditionality constant $C$.  
Because it is very difficult to choose the sequence $(d_j)_{j=1}^N$  in general, we will identify a situation where the optimal choice  is  $d_j=\overline{d_j}^{-1}=1$. That is, choosing $d_j=\overline{d_j}^{-1}=1$ will minimize the maximum of the Bessel bounds of $(d_j x_j)_{j=1}^N$ and $(\overline{d_j}^{-1} f_j)_{j=1}^N$.
We will make use of the following simple lemma. 

\begin{lem}\label{L:trace}
Let $(x_j)_{j=1}^N$ be a finite frame for $\ell_2^M$ with frame operator $S_X$, upper frame bound $B$, and lower frame bound $A$. Then the following hold:
\begin{enumerate}[(i)]
\item  $\emph{trace}(S_X)=\sum_{j=1}^N \|x_j\|^2$,\label{lem:eq1}
\item  $AM\leq \emph{trace}(S_X)\leq BM$,\label{lem:eq2}
\item If $(x_j)_{j=1}^N$  is $B$-tight, then $\emph{trace}(S_X)=BM$.\label{lem:eq3}
\item Let $U_X$ be the analysis operator of $(x_j)_{j=1}^N$ which has the following matrix form,
\[ U_X=
\begin{bmatrix}
    -      & x_1 & -  \\
    -      & x_2 & -  \\
 & \vdots & \\
    -      & x_N & -  \\
\end{bmatrix}_{N\times M}
=\begin{bmatrix}
    |     & | &   &  | \\
    c_1  & c_2  &\hdots & c_M \\
    |     & | &   &  | \\
\end{bmatrix}_{N\times M}
\]\label{lem:eq4}
\end{enumerate}
Then, the sequence of columns $(c_j)_{j=1}^M$ of $U_X$ is a basis for the column space and has upper frame bound $B$ and lower frame bound $A$.
\end{lem}
\begin{proof}
The operator $S_X:\ell_2^M\rightarrow \ell_2^M$ is defined by $S_X(x)=\sum_{j=1}^N \langle x, x_j\rangle x_j$ for all $x\in \ell_2^M$.
Let $(e_k)_{k=1}^M$ be an orthonormal basis for $\ell_2^M$. Then,
$$
\text{trace}(S_X)=\sum_{k=1}^M\sum_{j=1}^N |\langle e_k,x_j\rangle|^2=\sum_{j=1}^N \sum_{k=1}^M |\langle e_k,x_j\rangle|^2=\sum_{j=1}^N \|x_j\|^2.
$$
This gives \eqref{lem:eq1}.  Conditions 
\eqref{lem:eq2} and \eqref{lem:eq3} follow easily (see \cite[Cor. 5.2]{B2}). 

Note that the upper frame bound $B$ is the largest eigenvalue of the frame operator $S_X=U_X^*U_X$ and is hence the square of the largest singular value of $U_X$.  As $U_X$ and $U_X^*$ have the same non-zero singular values, $B$ is the upper frame bound of $(c_j)_{j=1}^M$.  Likewise, $A$ is the lower frame bound of $(c_j)_{j=1}^M$.
\end{proof}

\noindent The following proposition gives a case where the optimal choice for $(d_j)_{j=1}^N$ is constant $1$. 
That is, we provide a situation where the local optimization of having $\|d_j x_j\|=\|d_j^{-1} f_j\|$ gives the global optimization of minimizing $\max\{B_{dX},B_{d^{-1}F}\}$ where $B_{dX}$ is the Bessel bound of $(d_j x_j)_{j=1}^N$ and $B_{d^{-1}F}$ is the Bessel bound of $(d_j^{-1} f_j)_{j=1}^N$.  Furthermore, this minimizes the condition numbers for both $(d_j x_j)_{j=1}^N$ and $(d_j^{-1} f_j)_{j=1}^N$ to be $1$ as in the sense of weighted frames \cite{BAG}.

\begin{prop}\label{P:1}
Let $(x_j)_{j=1}^N$ and $(f_j)_{j=1}^N$ be finite frames for $\ell_2^M$ with $\|x_j\|=\|f_j\|$, for all $1\leq j\leq N$. Let $A>0$ be a lower frame bound for either $(x_j)_{j=1}^N$ or $(f_j)_{j=1}^N$.  Then for all non-zero scalars $(d_j)_{j=1}^N$, if $B$ is a Bessel bound for both $(d_j x_j)_{j=1}^N$ and $(\frac{1}{d_j}f_j)_{j=1}^N$ then $B$ must be at least $A$.

In particular, we have that if $(x_j)_{j=1}^N$ and $(f_j)_{j=1}^N$ are both tight frames then they are both $A$-tight for some $A>0$ and
$$
A=\min_{d=(d_j)_{j=1}^N}\max\{B_{dX},B_{d^{-1}F}\},
$$
where $B_{dX}$ is the optimal Bessel bound of $(d_j x_j)_{j=1}^N$ and $B_{d^{-1}F}$ is the optimal Bessel bound of $(d_j^{-1} f_j)_{j=1}^N$. 
\end{prop}

\begin{proof}
Without loss of generality, we assume that $A$ is a lower frame bound for $(x_j)_{j=1}^N$.
Let $(d_j)_{j=1}^N$ be a sequence of non-zero scalars.  Let
$S_{dX}$ be the frame operator of $(d_j x_j)_{j=1}^N$ and $S_{d^{-1}F}$ be the frame operator of $(\frac{1}{d_j} f_j)_{j=1}^N$. We will prove that the  trace of $S_{dX}+S_{d^{-1}F}$ is at least $2AM$. Lemma~\ref{L:trace}~\eqref{lem:eq2} then gives that the minimal Bessel bound of
either $(d_j x_j)_{j=1}^N$ or $(\frac{1}{d_j}f_j)_{j=1}^N$ must be greater than or equal to $A$.

By Lemma \ref{L:trace}~\eqref{lem:eq1} we have that 
$$
\trace(S_{dX})=\sum_{j=1}^n |d_j|^2\|x_j\|^2,\qquad\textrm { and }\qquad
\trace(S_{d^{-1}F})=\sum_{j=1}^n |d_j|^{-2}\|f_j\|^2.
$$
By adding these equations together, we get that
\begin{align*}
\trace(S_{dX}+S_{d^{-1}F})&=\sum_{j=1}^N |d_j|^2\|x_j\|^2+|d_j|^{-2}\|f_j\|^2\\
&= \sum_{j=1}^N ( |d_j|^2+|d_j|^{-2})\|x_j\|^2 \qquad\textrm{ as }\|x_j\|=\|f_j\| \textrm{ for all }1\leq j\leq N,\\
&\geq  \sum_{j=1}^n 2\|x_j\|^2\qquad\textrm{as $t^2+t^{-2}$ is minimized  at $t= 1$,}\\
&=2\,\trace(S_X) \qquad \textrm {by Lemma~\ref{L:trace}~\eqref{lem:eq1},}\\
&\geq 2AM \qquad \textrm {by Lemma~\ref{L:trace}~\eqref{lem:eq2}.}
\end{align*}
Thus $\trace(S_{dX}+S_{d^{-1}F})\geq 2AM$ and our proof is complete.
\end{proof}

 Proposition \ref{P:1} gives a general situation where we know the optimal values for the sequence $(d_j)_{j=1}^N$.  That is, if $(x_j)_{j=1}^N$ and $(f_j)_{j=1}^N$ are both tight frames with $\|x_j\|=\|f_j\|$, for all $1\leq j\leq N$ then choosing $d_j=1$ for all $1\leq j\leq N$ will minimize the maximum of the Bessel bounds of $(d_j x_j)_{j=1}^N$ and $(d^{-1}_j f_j)_{j=1}^N$.
 Other than specific examples, this is the only general situation where the optimal values for $(d_j)_{j=1}^N$ are known, and we will solve Conjecture \ref{C:SB quant} completely for this case.  In particular,  Theorem \ref{T:main1} gives that if such frames $(x_j)_{j=1}^N$ and $(f_j)_{j=1}^N$ satisfy \eqref{E:unc} for some constant $C>0$ then $(x_j)_{j=1}^N$ and $(f_j)_{j=1}^N$ both have Bessel bound $\frac{27}{4}K_1^{-4} C$.
 Note that this case includes many important examples of frames, including the finite unit norm tight frames.
 A sequence $(x_j)_{j=1}^N$ in $\ell_2^M$ is called a {\em finite unit norm tight frame or FUNTF} if $(x_j)_{j=1}^N$ is a tight frame and $\|x_j\|=1$ for all $1\leq j\leq N$.  Finite unit norm tight frames (FUNTFs) were introduced in \cite{BF}, and are of particular interest in both theory and applications \cite{BH}\cite{CFK}\cite{CMS}.   If $(x_j)_{j=1}^N$ is a FUNTF for $\ell_2^M$ then its frame bound is exactly $N M^{-1}$.  An {\em equi-norm tight frame} is a tight frame where all the vectors have the same norm, and in finite dimensions is just a rescaling of a FUNTF.  This gives the following corollary for Proposition \ref{P:Par} for the case that $(x_j)_{j=1}^N$ and $(f_j)_{j=1}^N$ are frames which are both equi-norm and tight.

 \begin{cor}\label{C:Par}
 Let $C\geq1$ and $M,N\in \N$.  Suppose that $(x_j)_{j=1}^N$ and $(f_j)_{j=1}^N$ are both equi-norm tight frames with 
 $$\Big\|\sum_{j=1}^N \vp_j\langle x, f_j\rangle  x_j\Big\|\leq C\|x\|, \qquad\textrm{ for all $x\in \ell_2^M$ and $|\vp_j|=1$}.
$$
Then $(dx_j)_{j=1}^N$ and $(d^{-1}f_j)_{j=1}^N$ have Bessel bound $N^{1/2}M^{-1/2}C$ where $d=\|x_1\|^{-1/2}\|f_1\|^{1/2}$.
 \end{cor}
 
\begin{proof}
As  $(x_j)_{j=1}^N$ and $(f_j)_{j=1}^N$ are both equi-norm tight frames there exists a constant $b>0$ so that $\|x_j\|\|f_j\|=b$ for all $1\leq j\leq N$.  We apply Proposition \ref{P:Par} to obtain that the sequences $(dx_j)_{j=1}^N$ and  $(d^{-1}f_j)_{j=1}^N$ have Bessel bound $b^{-1}C^2$.  As $(b^{-1/2}dx_j)_{j=1}^N$ is a FUNTF, it has frame bound $NM^{-1}$.   By scaling, $(dx_j)_{j=1}^N$ is a tight frame with frame bound $b NM^{-1}$.  This gives that $bNM^{-1}\leq b^{-1}C^2$.  Hence, $b\leq N^{-1/2}M^{1/2} C$.  As $(dx_j)_{j=1}^N$ is a tight frame with frame bound $bNM^{-1}$ we get that $(dx_j)_{j=1}^N$ has Bessel bound $N^{1/2}M^{-1/2} C$.  Likewise, $(d^{-1}f_j)_{j=1}^N$ has Bessel bound $N^{1/2}M^{-1/2}C$.
 \end{proof}
 
One interesting aspect of Corollary \ref{C:Par} is that it gives a Bessel bound for the case of equi-norm tight frames which does not explicitly state the value of $\|x_j\|$ and $\|f_j\|$.   The Bessel bound given in Corollary \ref{C:Par} depends on both $N$ and $M$.  There exist length $N$ FUNTFs for $\ell_2^M$ for all $N\geq M$, and hence the Bessel bound $N^{1/2}M^{-1/2}C$ can be arbitrarily large. One of our goals for the remainder of this section is to improve this by giving a uniform Bessel bound which is independent of $N$ and $M$.   Our proof will be probabilistic and will rely on the following case of Khintchine's Inequality (see for example Lemma 6.29 in \cite{FHHSPZ} for the general statement).

\begin{thm}[Khintchine's Inequality]\label{thm:khintchine}
There exists a constant $K_{1}>0$ such that for all $M\in\N$ and all scalars $(a_j)_{j=1}^N$,
$$2^{-N}\sum_{\delta_j=\pm1} \Big|\sum_{j=1}^N \delta_j a_j \Big|\geq K_{1} \left(\sum_{j=1}^N |a_j|^2\right)^{1/2}.
$$
That is, if $(\delta_j)_{j=1}^N$ is a sequence of independent symmetric random variables with $\delta_j=\pm1$, then we have the following lower bound for the expectation,
$$ \mathbb{E}\Big|\sum_{j=1}^N \delta_j a_j\Big|\geq K_{1}\big\|(a_j)_{j=1}^N\big\|_{\ell_2^N}.$$
\end{thm}

We use the notation $K_{1}$ for the constant in Khintchine's inequality because it gives a bound on the $\ell_1$-norm.  Khintchine's inequality may be used to compare any pair of $\ell_p$ and $\ell_q$-norms for $1\leq p,q<\infty$ (with a constant which depends on $p,q$), but we will only need it for comparing the $\ell_1$ and $\ell_2$ norms.  We are now ready to prove the following solution to Conjecture~\ref{C:SB quant} for the case of families of equi-norm vectors.  We will extend this to more general sets of vectors in Theorem \ref{T:main1}.

\begin{lem} \label{L:equinorm}
 Let $C,D,\beta>0$ and $N\in\N$.  Suppose that $(x_j)_{j=1}^N$ and $(f_j)_{j=1}^N$ are both sequences in a finite dimensional Hilbert space  which satisfy $\|x_j\|=\|f_j\|=D$ for all $1\leq j\leq N$ and
 $$\Big\|\sum_{j=1}^N \vp_j\langle x, f_j \rangle  x_j\Big\|\leq C\|x\|, \qquad\textrm{ for all $x\in \ell_2^M$ and $|\vp_j|=1$}.
$$
If the frame operator for $(f_j)_{j=1}^N$  has eigenvalues $\lambda_1\geq...\geq\lambda_M$   satisfying $\lambda_1\leq  \frac{\beta}{M}\sum_{j=1}^M\lambda_j$ then $(f_j)_{j=1}^N$ has Bessel bound $\frac{27}{4} K_{1}^{-4} \beta^2 C$, where $K_{1}$ is the constant in Khintchine's inequality. 

Likewise, if the frame operator for $(x_j)_{j=1}^N$  has eigenvalues $\lambda_1\geq...\geq\lambda_M$ which satisfy $\lambda_1\leq \frac{\beta}{M}\sum_{j=1}^M\lambda_j$ then $(x_j)_{j=1}^N$ has Bessel bound $\frac{27}{4} K_{1}^{-4} \beta^2 C$. 
\end{lem}
\begin{proof}

Let $M\in\N$ and assume that $(x_j)_{j=1}^N$ and $(f_j)_{j=1}^N$ are sequences of vectors in $\ell_2^M$. Let $\beta>0$ be such that  $\lambda_1\leq \beta M^{-1}\sum_{j=1}^M\lambda_j$ where $\lambda_1\geq...\geq\lambda_M$ are the eigenvalues of the frame operator of $(f_j)_{j=1}^N$.
Let $[a_{i,j}]_{N\times M}$ be the analysis matrix for the frame $(f_j)_{j=1}^N$ and let $[b_{j,i}]_{M\times N}$ be the synthesis matrix for the frame $(x_j)_{j=1}^N$.  That is,
\[ 
\begin{bmatrix}
    -      & f_1 & -  \\
    -      & f_2 & -  \\
 & \vdots & \\
    -      & f_N & -  \\
\end{bmatrix}
=
\begin{bmatrix}
    a_{1,1} &  \dots  & a_{1,M} \\
    a_{2,1} &  \dots  & a_{2,M} \\
    \vdots & \ddots & \vdots \\
    a_{N,1} &   \dots  & a_{N,M}
\end{bmatrix}
,\quad \textrm{ and }\quad
\begin{bmatrix}
    |     & | &   &  | \\
    x_1  & x_2  &\hdots & x_N \\
    |     & | &   &  | \\
\end{bmatrix}
=
\begin{bmatrix}
    b_{1,1} &   \dots  & b_{1,N} \\
    b_{2,1} &  \dots  & b_{2,N} \\
    \vdots &  \ddots & \vdots \\
    b_{M,1} &   \dots  & b_{M,N}
\end{bmatrix}
\]
We now claim that the following two conditions are satisfied.

\begin{enumerate}[(i)]
\item  $(\sum_{j=1}^M a_{i,j}^2)^{1/2}=(\sum_{j=1}^M b_{j,i}^2)^{1/2}=D$  for all $1\leq i\leq N$,  \label{pr:i}
\item The sequence of columns of $[a_{i,j}]_{N\times M}$ has Bessel bound $\beta NM^{-1}D^2$.  \label{pr:ii}
\end{enumerate}

Note that \eqref{pr:i} is simply that $\|x_j\|=\|f_j\|=D$, for all $1\leq j\leq N$.  To prove \eqref{pr:ii} we note that $\lambda_1$ is the optimal Bessel bound of $(f_j)_{j=1}^N$ and hence $\lambda_1$ is the optimal Bessel bound of the  columns of $[a_{i,j}]_{N\times M}$ by Lemma \ref{L:trace}.  By taking the trace of the frame operator of $(f_j)_{j=1}^N$ we have that $\sum_{i=1}^M\lambda_i=\sum_{j=1}^N\|f_j\|^2=ND^2$.  As, $\lambda_1\leq \beta M^{-1}\sum_{j=1}^M\lambda_j$ we have that $\lambda_1\leq \beta NM^{-1}D^2$.  Thus, we have proven that \eqref{pr:ii} is true. The rest of the proof is concerned with bounding  $\beta NM^{-1}D^2$ in terms of $C$ and $\beta$.

 Let $(\delta_j)_{j=1}^M$  be a sequence of independent symmetric random variables with $\delta_j=\pm1$.
 By taking expectation we calculate,
\begin{align*}
\E \sum_{i=1}^N \Big|\sum_{j=1}^M \delta_j a_{i,j}\Big|&=  \sum_{i=1}^N \E \Big|\sum_{j=1}^M \delta_j a_{i,j}\Big|\\
&\geq K_{1} \sum_{i=1}^N \left(\sum_{j=1}^M |a_{i,j}|^2\right)^{1/2} \qquad\textrm{by Khintchine's Inequality,}\\
&= K_{1} \sum_{i=1}^N D= K_{1} D N \qquad\textrm{by \eqref{pr:i}.}
\end{align*}

\noindent Thus, we may fix a particular realization for  $(\delta_j)_{j=1}^M$ such that 
\begin{equation}\label{E:choice_1}
 \sum_{i=1}^N \Big|\sum_{j=1}^M \delta_j a_{i,j}\Big|\geq K_{1} D N.
\end{equation}

\noindent For each $\alpha>0$, we define a subset $I_\ap\subseteq \{1,2,...,N\}$ by,
$$I_\ap:= \Big\{i\in\{1,...,N\}\,:\,  \Big|\sum_{j=1}^M \delta_j a_{i,j}\Big|\geq D\ap  \Big\}.
$$
By \eqref{pr:ii}, the vector $(\sum_{j=1}^M \delta_j a_{i,j})_{i=1}^N $ is in the column space of $[a_{i,j}]_{N\times M}$ and has norm at most $\beta^{1/2} DN^{1/2}$. We will now calculate a lower bound for the cardinality of $I_\alpha$.

\begin{align*}
\beta^{1/2}DN^{1/2}&\geq\left(\sum_{i=1}^N \Big|\sum_{j=1}^M \delta_j a_{i,j}\Big|^2\right)^{1/2}\\
&\geq \left(\sum_{i\in I_\ap} \Big|\sum_{j=1}^M \delta_j a_{i,j}\Big|^2\right)^{1/2}\\
&\geq |I_\ap|^{-1/2}\sum_{i\in I_\ap} \Big|\sum_{j=1}^M \delta_j a_{i,j}\Big| \qquad\textrm{by Cauchy-Schwarz,}\\
&= |I_\ap|^{-1/2}\left(\sum_{i=1}^N \Big|\sum_{j=1}^M \delta_j a_{i,j}\Big|-\sum_{i\not\in I_\ap} \Big|\sum_{j=1}^M \delta_j a_{i,j}\Big|\right)\\
&\geq  |I_\ap|^{-1/2}\left(\sum_{i=1}^N \Big|\sum_{j=1}^M \delta_j a_{i,j}\Big|- (N-|I_\ap|)D\ap\right)\\
&\geq  |I_\ap|^{-1/2}\Big(K_{1} D N- (N-|I_\ap|)D\ap\Big)\qquad\textrm{ by \eqref{E:choice_1},}\\
&\geq  |I_\ap|^{-1/2} (K_{1}  -\ap)DN.
\end{align*}
Thus, we have that $|I_\ap|\geq \beta^{-1}(K_{1} -\ap)^2 N$.  We now apply similar estimates to the matrix $[b_{j,i}]_{M\times N}$ restricted to the columns in $I_\ap$.

Let $(\gamma_j)_{j=1}^M$  be a sequence of independent symmetric random variables with $\gamma_j=\pm1$.
 By taking the expectation we calculate,
\begin{align*}
\E \sum_{i\in I_\ap} \Big|\sum_{j=1}^M \gamma_j b_{j,i}\Big|&=  \sum_{i\in I_\ap} \E\Big|\sum_{j=1}^M \gamma_j b_{j,i}\Big|\\
&\geq K_{1} \sum_{i\in I_\ap} \left(\sum_{j=1}^M |b_{j,i}|^2\right)^{1/2} \qquad\textrm{by Khintchine's Inequality,}\\
&= K_{1} \sum_{i\in I_\ap} D \quad\textrm{ by (1),}\\
&=DK_{1}|I_\ap|\geq D \beta^{-1} K_{1} (K_{1} -\ap)^2 N.
\end{align*}

\noindent Thus, we may fix a particular realization for  $(\gamma_j)_{j=1}^M$ such that 
\begin{equation}\label{E:choice_2}
 \sum_{i\in I_\ap} \Big|\sum_{j=1}^M \gamma_j b_{i,j}\Big|\geq \beta^{-1} DK_{1}(K_{1}-\ap)^2 N.
\end{equation}
We have that $x:=(M^{-1/2} \delta_j)_{j=1}^M$ and $f:=(M^{-1/2}\gamma_j)_{j=1}^M$ are unit norm vectors in $\ell_2^M$.  For $1\leq i\leq N$, we define $$\vp_i:=\left\{\begin{array}{ll}\text{phase}( \langle x, f_i\rangle \langle x_i,f\rangle )^{-1}, & \textrm{ if } \langle x, f_i\rangle \langle x_i,f\rangle \neq 0,\\1, &\textrm{ otherwise.}
\end{array}\right. $$

Thus, we have that 

\begin{align*}
C\|x\|&\geq\Big\|\sum_{i=1}^N \vp_i \langle x, f_i\rangle x_i\Big\|\\
&\geq  \sum_{i=1}^N \vp_i \langle x_i, f\rangle \langle x, f_i\rangle\qquad \textrm{ as $\|f\|=1$,}\\
&= \sum_{i=1}^N |\langle x_i, f\rangle \langle x, f_i\rangle|\\
&= M^{-1} \sum_{i=1}^N \Big|\sum_{j=1}^M \delta_j a_{i,j}\Big| \Big|\sum_{j=1}^M \gamma_j b_{j,i}\Big|\\
&\geq M^{-1} \sum_{i\in I_\ap} \Big|\sum_{j=1}^M \delta_j a_{i,j}\Big| \Big|\sum_{j=1}^M \gamma_j b_{j,i}\Big|\\
&\geq M^{-1} \sum_{i\in I_\ap} \ap D \Big|\sum_{j=1}^M \gamma_j b_{j,i}\Big|\\
&\geq  \ap D^2 \beta^{-1} K_{1}(K_{1}-\ap)^2 NM^{-1}\qquad\qquad\textrm{ by \eqref{E:choice_2},}\\
&= \frac{4}{27} D^2 \beta^{-1} K_{1}^4 NM^{-1}\qquad \textrm{ for }\ap=K_{1}/3.
\end{align*}
Thus,  $\beta D^2NM^{-1}\leq \frac{27}{4} \beta^2  CK_{1}^{-4}$ as $\|x\|=1$. From $(ii)$ we deduce that $(f_j)_{j=1}^N$   has Bessel bound $\frac{27}{4} \beta^2 CK_{1}^{-4}$.  
\end{proof}

In Lemma \ref{L:equinorm} we considered frames $(x_j)_{j=1}^N$ and $(f_j)_{j=1}^N$ such that $\|x_j\|=\|f_j\|=D$ for all $1\leq j\leq N$.  We will now extend this to  the case where $\|x_j\|=\|f_j\|$ for all $1\leq j\leq N$ but we no longer require that $\|x_j\|= \|x_i\|$ when $i\neq j$.  The following theorem is a restatement of Theorem \ref{T:intro} from the introduction along with some immediate corollaries.

\begin{thm} \label{T:main1}
 Let $C,D,\beta>0$ and $N\in\N$.  Suppose that $(x_j)_{j=1}^N$ and $(f_j)_{j=1}^N$ are both sequences in a Hilbert space $H$ which satisfy $\|x_j\|=\|f_j\|$ for all $1\leq j\leq N$ and
 $$\Big\|\sum_{j=1}^N \vp_j\langle x, f_j\rangle  x_j\Big\|\leq C\|x\|, \qquad\textrm{ for all $x\in H$ and $|\vp_j|=1$}.
$$
If the frame operator for $(f_j)_{j=1}^N$  has eigenvalues $\lambda_1\geq...\geq\lambda_M$  satisfying $\lambda_1\leq \frac{\beta}{M}\sum_{j=1}^M\lambda_j$ then $(f_j)_{j=1}^N$ has Bessel bound $\frac{27}{4} K_{1}^{-4} \beta^2 C$, where $K_{1}$ is the constant in Khintchine's inequality. 
In particular, if $(f_j)_{j=1}^N$ has condition number $\beta$ then $(f_j)_{j=1}^N$ has Bessel bound $\frac{27}{4} K_{1}^{-4} \beta^2 C$ and if $(f_j)_{j=1}^N$ is a tight frame then $(f_j)_{j=1}^N$ has Bessel bound $\frac{27}{4} K_{1}^{-4} C$. 

Likewise, if the frame operator for $(x_j)_{j=1}^N$  has eigenvalues $\lambda_1\geq...\geq\lambda_M$ which satisfy $\lambda_1\leq \frac{\beta}{M}\sum_{j=1}^M\lambda_j$ then $(x_j)_{j=1}^N$ has Bessel bound $\frac{27}{4} K_{1}^{-4} \beta^2 C$. 
\end{thm}

\begin{proof}
Let $\beta>0$ be such that  $\lambda_1\leq \beta M^{-1}\sum_{j=1}^M\lambda_j$ where $\lambda_1\geq...\geq\lambda_M$ are the eigenvalues of the frame operator of $(f_j)_{j=1}^N$.

We first will assume that the values $\|x_j\|^2=\|f_j\|^2$ are non-zero and rational for all $1\leq j\leq N$.
Thus, there exists $K,k_1,...,k_N\in \N$ such that $\|x_j\|^2=\|f_j\|^2=\frac{k_j}{K}$ for all $1\leq j\leq N$.  We have that 
$$
\sum_{j=1}^N\sum_{i=1}^{k_j}\langle x,k_j^{-1/2}f_j\rangle k_j^{-1/2}x_j=\sum_{j=1}^N\langle x,f_j\rangle x_j,\hspace{1cm}\textrm{ for all }x\in H.
$$
Thus, the systems $\Big(\Big(\sqrt{\frac{1}{k_j}}x_j\Big)_{i=1}^{k_j}\Big)_{j=1}^N$ and $\Big(\Big(\sqrt{\frac{1}{k_j}}f_j\Big)_{i=1}^{k_j}\Big)_{j=1}^N$ are each equi-norm  sequences with the same frame operator  as $(x_j)_{j=1}^N$ and $(f_j)_{j=1}^N$ respectively.

We now claim for all  $x\in H$  that
\begin{equation}\label{E:claim}
\Big\|\sum_{j=1}^N \sum_{i=1}^{k_j} \vp_{i,j} \langle x,k_j^{-1/2} f_j\rangle k_j^{-1/2} x_j\Big\|\leq C\|x\|,\qquad\textrm{for all }|\vp_{i,j}|=1. 
\end{equation}
Assuming the claim, we have by Lemma \ref{L:equinorm} that $\Big(\Big(\sqrt{\frac{1}{k_j}}f_j\Big)_{i=1}^{k_j}\Big)_{j=1}^N$ and consequently also $(f_j)_{j=1}^N$ have Bessel bound $\frac{27}{4} K_{1}^{-4} \beta^2 C$. 

We now prove the claim that \eqref{E:claim} is satisfied.  Let $x\in H$.  We will first prove that  the set $\left\{\sum_{j=1}^N a_j \langle x,f_j\rangle x_j:\ |a_j|\leq 1\right\}$ is contained in   $\text{conv}\left\{\sum_{j=1}^N \varepsilon_j \langle x,f_j\rangle x_j:\ |\varepsilon_j|= 1\right\},$ the convex hull of the vectors $\sum_{j=1}^N \varepsilon_j \langle x,f_j\rangle x_j$.  To see this, let $(a_j)_{j=1}^N$ be a sequence of scalars and assume without loss of generality that $0\leq |a_1|\leq |a_2|\leq...\leq |a_N|\leq 1$.  We construct a sequence $(b_i)_{i=1}^{2N+1}$ in $[0,1]$ by $b_1=|a_1|$, $b_2=b_3=(|a_2|-|a_1|)/2$, $b_4=b_5=(|a_3|-|a_2|)/2$,..., $b_{2N}=b_{2N+1}=(1-|a_N|)/2$.  For each $1\leq j\leq N$ and $1\leq i\leq 2N+1$ we now give $\vp_{j,i}$ with $|\vp_{j,i}|=1$.   For $1\leq i\leq 2j-1$ we let $\vp_{j,i}=\text{phase}(a_j)$, and for $2j\leq i\leq 2N+1$ we let $\vp_{j,i}=(-1)^i$.  The telescoping aspect of the construction gives that $\sum_{i=1}^{2N+1} b_i=1$.  Furthermore,  $\sum_{i=1}^{2N+1}b_i\epsilon_{j,i}=a_j$, for all $1\leq j\leq N$.    We have that
$$
\sum_{i=1}^{2N+1} b_i\sum_{j=1}^N \varepsilon_{j,i} \langle x,f_j\rangle x_j =\sum_{j=1}^N\sum_{i=1}^{2N+1} b_i \varepsilon_{j,i} \langle x,f_j\rangle x_j=\sum_{j=1}^Na_j   \langle x,f_j\rangle x_j.
$$
This proves that $\left\{\sum_{j=1}^N a_j \langle x,f_j\rangle x_j:\ |a_j|\leq 1\right\}\subseteq \textrm{conv}\left\{\sum_{j=1}^N \varepsilon_j \langle x,f_j\rangle x_j:\ |\varepsilon_j|= 1\right\}$. 

Moreover, every vector $\sum_{i=1}^{L} b_i\sum_{j=1}^N \varepsilon_{j,i} \langle x,f_j\rangle x_j$ in  $\textrm{conv}\left\{\sum_{j=1}^N \varepsilon_j \langle x,f_j\rangle x_j:\ |\varepsilon_j|=1\right\}$ has the following bound on its norm.
\begin{align*}
\left\|\sum_{i=1}^{L} b_i\sum_{j=1}^N \varepsilon_{j,i} \langle x,f_j\rangle x_j \right\|&\leq 
\sum_{i=1}^{L} b_i\left\|\sum_{j=1}^N \varepsilon_{j,i} \langle x,f_j\rangle x_j \right\|\\ &\leq \sum_{i=1}^L b_i C\|x\|=C\|x\|.
\end{align*}
Thus, we have for all sequences $(a_j)_{j=1}^N$ with $|a_j|\leq 1$  that $\|\sum_{j=1}^N a_j \langle x, f_j\rangle x_j\|\leq C\|x\|$.  To prove our claim that \eqref{E:claim} is satisfied we let $a_j=\sum_{i=1}^{k_j}\vp_{j,i} k_j^{-1}$, for all $1\leq j\leq N$.  Note that $|a_j|\leq 1$ for all $1\leq j\leq N$.  Thus we have that
\begin{align*}
\Big\|\sum_{j=1}^N\sum_{i=1}^{k_j} \vp_{j,i} \langle x, k_j^{-1/2} f_j\rangle k_j^{-1/2}x_j\Big\|& = \Big\|\sum_{j=1}^N\Big(\sum_{i=1}^{k_j} \vp_{j,i} k_j^{-1}\Big)\langle x,  f_j\rangle x_j\Big\|\\
&=\Big\|\sum_{j=1}^N a_j \langle x, f_j\rangle x_j\Big\|
\leq C\|x\|.
\end{align*}
Thus, we have proven our claim that \eqref{E:claim} is satisfied and the proof is complete for the case that $\|x_j\|^2=\|f_j\|^2$ are non-zero and rational for all $1\leq j\leq N$.  

We now consider the general case where  $\|x_j\|^2=\|f_j\|^2$ for all $1\leq j\leq N$, but they may not all be rational.  We may throw out any terms which are zero, and thus we assume without loss of generality that $\|x_j\|^2=\|f_j\|^2$ are non-zero  for all $1\leq j\leq N$.  Let $\vp>0$ and let $\lambda_1\geq...\geq\lambda_M$ be the eigenvalues for the frame operator of $(f_j)_{j=1}^N$.  We may choose scalars $(\alpha_j)_{j=1}^N$ which are arbitrarily close to $1$ so that $\|\alpha_j x_j\|^2=\|\alpha_j f_j\|^2$ is rational for all $1\leq j\leq N$  and 
 $$\Big\|\sum_{j=1}^N \vp_j\langle \alpha_j f_j, x\rangle  \alpha_j x_j\Big\|\leq (C+\vp)\|x\|, \qquad\textrm{ for all $x\in H$ and $|\vp_j|=1$}.
$$
Furthermore, we may assume that if $\gamma_1\geq...\geq\gamma_M$ are the eigenvalues of the frame operator for $(\alpha_j f_j)_{j=1}^n$ then $(1+\vp)^{-1}\gamma_i\leq \lambda_i\leq (1+\vp)\gamma_i$ for all $1\leq i\leq M$.
Hence, by our previous argument,  $(\alpha_j f_j)_{j=1}^n$ has Bessel bound 
$\frac{27}{4} K_{1}^{-4} \beta^2(1+\vp)^4 (C+\vp)$.  As the scalars $(\alpha_j)_{j=1}^N$ may be chosen arbitrarily close to $1$ we have that $(f_j)_{j=1}^N$ has Bessel bound 
$\frac{27}{4} K_{1}^{-4} \beta^2(1+\vp)^4 (C+\vp)$.  As $\vp>0$ was arbitrary,  $(f_j)_{j=1}^N$ has Bessel bound 
$\frac{27}{4} K_{1}^{-4} \beta^2 C$.
\end{proof}


\end{document}